\documentclass[a4paper,reqno,oneside]{amsart}

\usepackage{amssymb,amsmath}

\usepackage[margin=3cm]{geometry}

\usepackage{paralist,enumitem}
\setlist{listparindent=0pt,parsep=3pt}
\setdefaultenum{1.}{}{}{}

\usepackage[pdftex,linktocpage,pdfstartview=FitH,colorlinks=true,allcolors=blue]{hyperref}
\usepackage{amsthm} 
\usepackage[capitalise]{cleveref}
\usepackage[initials,msc-links]{amsrefs}[2007/10/22]

\newtheorem{thm}{Theorem}[section]
\newtheorem{prop}[thm]{Proposition}
\newtheorem{cor}[thm]{Corollary}
\newtheorem{lem}[thm]{Lemma}
\newtheorem*{proposition*}{Proposition}
\newtheorem*{theorem*}{Theorem}
\theoremstyle{definition}

\theoremstyle{remark}
\newtheorem*{rem}{Remark}
\newtheorem*{rems}{Remarks}

\numberwithin{equation}{section}

\newcommand{\R}{\mathbb{R}}
\newcommand{\C}{\mathbb{C}}

\renewcommand{\P}{\mathbb{P}}

\DeclareMathOperator{\rank}{rank}

\newcommand{\Span}[1]{\operatorname{span}\{#1\}}

\DeclareSymbolFont{script}{U}{eus}{m}{n}
\DeclareSymbolFontAlphabet{\mathscr}{script}
\DeclareMathSymbol{\EuWedge}{0}{script}{"5E}

\renewcommand{\d}{\operatorname{d}}

\newcommand{\sub}{\subseteq}
\newcommand{\st}{\mathrel{|}}
\newcommand{\set}[1]{\{#1\}}

\newcommand{\half}{\tfrac12}

\newcommand{\triv}{\underline{\R}^{n+1,1}}
\newcommand{\ctriv}{\underline{\C}^{n+2}}
\newcommand{\cL}{\mathcal{L}}
\newcommand{\cD}{\mathcal{D}}
\newcommand{\cN}{\mathcal{N}}
\newcommand{\restr}[1]{{}_{|#1}}

\newcommand{\TitleWithDoi}[1]{\IfEmptyBibField{doi}{\textit{#1}}%
  {\href{https://doi.org/\BibField{doi}}{\textit{#1}}}}
\renewcommand{\eprint}[1]{\IfEmptyBibField{url}{\url{#1}}%
  {\href{\BibField{url}}{\url{#1}}}}

\BibSpec{article}{%
    +{}  {\PrintAuthors}                {author}
    +{,} { \TitleWithDoi}               {title}
    +{.} { }                            {part}
    +{:} { \textit}                     {subtitle}
    +{,} { \PrintContributions}         {contribution}
    +{.} { \PrintPartials}              {partial}
    +{,} { }                            {journal}
    +{}  { \textbf}                     {volume}
    +{}  { \PrintDatePV}                {date}
    +{,} { \issuetext}                  {number}
    +{,} { \eprintpages}                {pages}
    +{,} { }                            {status}
    +{,} { available at \eprint}        {eprint}
    +{}  { \parenthesize}               {language}
    +{}  { \PrintTranslation}           {translation}
    +{;} { \PrintReprint}               {reprint}
    +{.} { }                            {note}
    +{.} {}                             {transition}
    +{}  {\SentenceSpace \PrintReviews} {review}
}

\BibSpec{book}{%
    +{}  {\PrintPrimary}                {transition}
    +{,} { \TitleWithDoi}               {title}
    +{.} { }                            {part}
    +{:} { \textit}                     {subtitle}
    +{,} { \PrintEdition}               {edition}
    +{}  { \PrintEditorsB}              {editor}
    +{,} { \PrintTranslatorsC}          {translator}
    +{,} { \PrintContributions}         {contribution}
    +{,} { }                            {series}
    +{,} { \voltext}                    {volume}
    +{,} { }                            {publisher}
    +{,} { }                            {organization}
    +{,} { }                            {address}
    +{,} { \PrintDateB}                 {date}
    +{,} { }                            {status}
    +{}  { \parenthesize}               {language}
    +{}  { \PrintTranslation}           {translation}
    +{;} { \PrintReprint}               {reprint}
    +{.} { }                            {note}
    +{.} {}                             {transition}
    +{}  {\SentenceSpace \PrintReviews} {review}
}

\title{On conformal Gauss maps}
\author{F.E. Burstall}
\email{feb@maths.bath.ac.uk}
\address{Dept.\ of Mathematical Sciences\\University of Bath\\
  Bath BA2 7AY\\ UK.}

\subjclass[2010]{53A30 (primary), 53C43 (secondary)}

\begin{document}
\begin{abstract}
  We characterise the maps into the space of $2$-spheres in
  $S^n$ that are the conformal Gauss maps of conformal
  immersions of a surface into $S^n$.  In particular, we give an
  invariant formulation and efficient proof of a
  characterisation, due to Dorfmeister--Wang
  \cites{DorWan13,DorWan}, of the harmonic maps that are conformal
  Gauss maps of Willmore surfaces.
\end{abstract}
\maketitle

\section*{Introduction}
\label{sec:introduction}
A useful tool in conformal surface geometry is the
\emph{central sphere congruence} \citelist{\cite{Bla29}*{\S67}\cite{Tho24}} or
\emph{conformal Gauss map} \cite{Bry84}.  Geometrically,
the central sphere congruence of a surface in the conformal
$n$-sphere attaches to each point of the surface a
$2$-sphere, tangent to the surface at that point and having the same mean
curvature vector as the surface at that
point.  The space of $2$-spheres in $S^n$ may be identified
with the Grassmannian of $(3,1)$-planes in $\R^{n+1,1}$ and
so the central sphere congruence may be viewed as a map, the
conformal Gauss map, to this Grassmannian.

The utility of this construction is that it links the
(parabolic) conformal geometry of the sphere to the
(reductive) pseudo-Riemannian geometry of the Grassmannian.
For example, a surface is Willmore if and only if the
conformal Gauss map is harmonic
\citelist{\cite{Bla29}*{\S81}\cite{Bry84}\cite{Eji88}\cite{Rig87}}.
In another direction, away from umbilic points, the metric
induced by the conformal Gauss map, which is in the
conformal class of the surface, is invariant by conformal
diffeomorphisms of $S^n$ and even arbitrary rescalings of
the ambient metric \cites{Fia44,Wan98}.

The purpose of this short note is to characterise those maps
into the Grassmannian which are the conformal Gauss map of a
conformal immersion.  In so doing, we build on a result of
Dorfmeister--Wang \cites{DorWan13,DorWan} which treats the
case where the map is harmonic.  As a by-product of our
analysis, we give an invariant formulation and efficient
proof of their result.

It is a pleasure to thank David Calderbank, Udo
Hertrich-Jeromin and Franz Pedit for their careful reading
of and helpful comments on a previous draft of this paper.

\section{The conformal Gauss map}
\label{sec:conformal-gauss-map}

We view the conformal $n$-sphere $S^n$ as the projective
lightcone $\P(\cL)$ of $\R^{n+1,1}$
\citelist{\cite{Dar72a}*{Livre~II,
    Chapitre~VI}\cite{Her03}*{Chapter~1}}.  Here
$\cL=\set{v\in\R^{n+1,1}_{\times}\st (v,v)=0}$ and $(\,,\,)$
is the signature $(n+1,1)$ inner product.

Let $f:\Sigma\to S^n=\P(\cL)$ be a conformal immersion of a
Riemann surface into the conformal $n$-sphere.
Equivalently, $f$ is a null line subbundle of the trivial
bundle $\triv:=\Sigma\times\R^{n+1,1}$.

Define $f^{1,0}\leq\ctriv$ by
\begin{equation*}
  f^{1,0}=\Span{\sigma,\d_Z\sigma\st\sigma\in\Gamma f,Z\in T^{1.0}\Sigma}.
\end{equation*}
Here the notation $U\leq V$ means $U$ is a subbundle of $V$.
That $f$ is a conformal immersion is equivalent to $f^{1,0}$
being a rank $2$ isotropic subbundle of $\underline{\C}^{n+2}$.
Set $f^{0,1}:=\overline{f^{1,0}}$ and note that\footnote{We
  make no notational distinction between a real bundle and
  its complexification} $f^{1,0}\cap f^{0,1}=f$.

The \emph{conformal Gauss map of $f$} is the bundle of
$(3,1)$-planes $V\leq\triv$ given by
\begin{equation*}
  V=\Span{\sigma,\d_Z\sigma,\d_{\bar{Z}}\sigma,\d_Z\d_{\bar{Z}}\sigma
    \st\sigma\in\Gamma f,Z\in \Gamma T^{1,0}\Sigma}.
\end{equation*}

We have a decomposition $\triv=V\oplus V^{\perp}$ which
induces a decomposition of the flat connection $\d$:
\begin{equation*}
  \d=\cD+\cN,
\end{equation*}
where $\cD$ is a metric connection preserving $V$ and
$V^{\perp}$ while $\cN$ is a $1$-form taking values in
skew-endomorphisms of $\triv$ which \emph{permute} $V$ and
$V^{\perp}$.
\begin{rem}
  We may view $V$ as a map from $\Sigma$ into the
  Grassmannian of $(3,1)$-planes in $\R^{n+1,1}$ and then
  $\cN$ can be identified with its differential.
\end{rem}
The flatness of $\d$ yields the structure equations of the
situation:
\begin{subequations}
  \label{eq:1}
  \begin{align}
    \label{eq:2}
    0&=R^{\cD}+\half[\cN\wedge\cN]\\
    0&=\d^{\cD}\cN.\label{eq:3}
  \end{align}
\end{subequations}
Here $\d^{\cD}$ is the exterior derivative on bundle-valued
forms with $\cD$ used to differentiate coefficients.

The conformal Gauss map $V$ is defined by the following
properties:
\begin{compactenum}
\item $f^{0,1}\leq V$;
\item $f^{0,1}\leq\ker\cN^{1,0}$.
\end{compactenum}
Now, for $Z$ a local section of $T^{1,0}\Sigma$,
$\cN_Z$ is skew while $f^{0,1}$ is maximal isotropic
in $V$ so that
\begin{equation}\label{eq:4}
  \cN_{Z}V^{\perp}=(\ker\cN_{Z}\restr{V})^{\perp}
  \sub (f^{0,1})^{\perp}\cap V=f^{0,1}\leq\ker\cN_{Z}\restr{V} 
\end{equation}
and we conclude:
\begin{lem}[c.f.\ \cite{DorWan13}*{Proposition~2.2}]\label{th:1}
  If $V$ is the conformal Gauss map of a conformal immersion
  then $(\cN^{1,0})^2\restr{V^{\perp}}=0$.
\end{lem}
Following \cite{DorWan13}, we say that $V$ with the property
of \cref{th:1} is \emph{strongly conformal}.

The conformal Gauss map also satisfies a second order
condition.  First note that \eqref{eq:4} tells us that
\begin{equation}
  \label{eq:5}
  \cN_{Z}V^{\perp}\sub f^{0,1}.
\end{equation}
Moreover, $f^{0,1}$ is $\cD^{0,1}$-stable thanks to the
following lemma which will see further use in
\cref{sec:converse-harmonic-v}:
\begin{lem}
  \label{th:2}
  Let $W\leq V$ be maximal isotropic in $V$ with a
  never-vanishing section $w$ such that $\cD_{\bar{Z}}w\in
  W$, for $\bar{Z}\in T^{0,1}\Sigma$.  Then $W$ is
  $\cD^{0,1}$-stable.
\end{lem}
\begin{proof}
  Let $u\in\Gamma W$ be a local section so that $u,w$
  locally frame $W$.  It suffices to show that
  $\cD_{\bar{Z}}u\in W$.  However,
  \begin{align*}
    (\cD_{\bar{Z}}u,w)&=-(u,\cD_{\bar{Z}}w)=0;\\
    (\cD_{\bar{Z}}u,u)&=\half \bar{Z}(u,u)=0,
  \end{align*}
  since $W$ is isotropic.  Thus $\cD_{\bar{Z}}u\in
  W^{\perp}\cap V=W$ since $W$ is maximal isotropic in $V$.
\end{proof}
In the case at hand, for $\sigma\in\Gamma f$, we have
$\cD_{\bar{Z}}\sigma\in f^{0,1}$ by definition so
\cref{th:2} applies to show that $f^{0,1}$ is
$\cD^{0,1}$-stable.

Now contemplate the \emph{tension field}
$\tau_{V}:=*\d^{\cD}*\cN$ of $V$.  Since
$*\cN=i(\cN^{1,0}-\cN^{0,1})$, \eqref{eq:3} yields
\begin{equation*}
  \tau_V=2i*\d^{\cD}\cN^{1,0}=-2i*\d^{\cD}\cN^{0,1}.
\end{equation*}
In view of the last paragraph, $\tau_VV^{\perp}$ takes
values in $f^{0,1}$ since $*\d^{\cD}\cN^{1,0}V^{\perp}$
does.  However, $\tau_V$ is real so that $\tau_VV^{\perp}$
takes values in $f^{0,1}\cap f^{1,0}=f$:
\begin{equation}
  \label{eq:6}
  \tau_VV^{\perp}\sub f.
\end{equation}
In particular, since $f$ is a null line subbundle on which
$\cN$ vanishes, we conclude:
\begin{prop}
  \label{th:3}
  If $V$ is the conformal Gauss map of a conformal immersion
  with tension field $\tau_{V}$.  Then:
  \begin{subequations}\label{eq:7}
    \begin{align}
      \cN\circ\tau_V\restr{V^{\perp}}&=0\label{eq:8}\\
      (\tau_V)^2\restr{V^{\perp}}&=0\label{eq:9}.
    \end{align}
  \end{subequations}
\end{prop}

This line of argument additionally give us some control on
the rank of $\cN:T\Sigma\otimes V^{\perp}\to V$:
\begin{lem}
  \label{th:4}
  Let $V$ be the conformal Gauss map of a conformal
  immersion $f$, then the set
  \begin{equation*}
    A:=\set{p\in\Sigma\st
      \cN_ZV^{\perp}=\cN_{\bar{Z}}V^{\perp}\neq\set0, Z\in
      T_p\Sigma}
  \end{equation*}
  is nowhere dense.
\end{lem}
\begin{proof}
  Any open set in the closure of $A$ must contain an open
  set where
  $\cN_ZV^{\perp}=\cN_{\bar{Z}}V^{\perp}\neq\set0$.  On this
  latter set, we immediately see from \eqref{eq:5} that
  $\cN_ZV^{\perp}=f$.  Since $\tau_VV^{\perp}\leq f$ also,
  by \eqref{eq:6}, we rapidly conclude (c.f.\ \cref{th:6}
  below) that $f$ is $\cD^{0,1}$-stable and so $\cD$-stable.
  Since $\cN f=0$ also, $f$ is constant: a contradiction.
\end{proof}

In the next section, we will establish a generic converse to
these results.

\section{Reconstruction of $f$ from $V$}
\label{sec:converse-harmonic-v}

Suppose now that we have a bundle $V\leq\triv$ of
$(3,1)$-planes and ask whether $V$ is the conformal Gauss
map of some conformal immersion $f$.  Our task is then
to construct $f\leq V$ but, in fact, it will be more
convenient to construct $f^{0,1}$:
\begin{prop}
  \label{th:5}
  Let $W\leq V$ be a maximal isotropic subbundle of $V$ such
  that:
  \begin{compactenum}
  \item $W$ is $\cD^{0,1}$-stable;
  \item $\cN^{1,0}W=0$, or, equivalently (c.f. \eqref{eq:4}),
    $\cN_{Z}V^{\perp}\sub W$, for all $Z\in T^{1,0}\Sigma$.
  \end{compactenum}
  Then $f:=W\cap \overline{W}$ is a real, null, line subbundle
  which, on the open set where it immerses, is conformal
  with $W=f^{0,1}$ and conformal Gauss map $V$.
\end{prop}
\begin{proof}
  Since $V$ has signature $(3,1)$, $W$ and $\overline{W}$
  must intersect in a line bundle, necessarily null and
  real.  Since $f$  is real,
  $f\leq\ker\cN^{1,0}\cap\ker\cN^{0,1}$ so that, for
  $\sigma\in\Gamma f$, $\bar{Z}\in T^{0,1}\Sigma$,
  \begin{equation*}
    \d_{\bar{Z}}\sigma=\cD_{\bar{Z}}\sigma+\cN_{\bar{Z}}\sigma=
    \cD_{\bar{Z}}\sigma\in W,
  \end{equation*}
  since $f\leq W$ and $W$ is $\cD^{0,1}$-stable.  Thus
  $W=f^{0,1}$ on the set where $f$ immerses. We conclude
  that, on that set, $f$ is conformal, since $f^{0,1}$ is
  isotropic and $V$ is the conformal Gauss map of $f$ since
  $f^{0,1}\leq\ker\cN^{1,0}$.
\end{proof}

For our main result, we need the following simple
observation:
\begin{lem}
  \label{th:6}
  Let $V\leq\triv$ be a bundle of $(3,1)$-planes with
  tension field $\tau_V$.
  Let $w=\cN_{Z}\nu$, for $\nu\in\Gamma V^{\perp}$ and $Z\in
  T\Sigma$.  Then
  $\cD_{\bar{Z}}w\in \cN_ZV^{\perp}+\tau_VV^{\perp}$. 
\end{lem}
\begin{proof}
  For suitable $Z\in T^{1,0}\Sigma$,
  $\tau_V=\cD_{\bar{Z}}\cN_Z-\cN^{1,0}_{[\bar{Z},Z]}$ so
  that
  \begin{multline*}
    \cD_{\bar{Z}}w=\cD_{\bar{Z}}(\cN_Z\nu)=
    (\cD_{\bar{Z}}\cN_Z)\nu+\cN_Z(\cD_{\bar{Z}}\nu)\\
    =\tau_{V}\nu+\cN^{1,0}_{[\bar{Z},Z]}\nu+\cN_Z(\cD_{\bar{Z}}\nu)
    \in\cN_ZV^{\perp}+\tau_VV^{\perp}.
  \end{multline*}
\end{proof}

With all this in hand, we have:
\begin{thm}
  \label{th:7}
  Let $V\leq\triv$ be a bundle of $(3,1)$-planes with
  tension field $\tau_V$.  Suppose that:
  \begin{compactenum}
  \item $V$ is strongly conformal.\label{item:1}
  \item Equations \eqref{eq:7} hold.\label{item:2}
  \item $\set{p\in\Sigma\st
      \cN_ZV^{\perp}=\cN_{\bar{Z}}V^{\perp}\neq\set0, Z\in
      T_p\Sigma}$ is empty.\label{item:3}
  \end{compactenum}
  Set $U:=\cN_ZV^{\perp}+\tau_VV^{\perp}$ and restrict
  attention to the open dense subset of $\Sigma$ where $U$
  has fibres of locally constant dimension and so is a
  vector bundle.

  Then $\rank U\leq 2$ and we have:
  \begin{compactenum}[(a)]
  \item Where $\rank U=2$, there is a unique real, null line
    subbundle $f\leq V$ which, where it immerses, is a
    conformal immersion with conformal Gauss map $V$.
  \item Where $\rank U=1$, there are exactly two real, null
    line subbundles $f,\hat{f}\leq V$, which, where they
    immerse, are conformal immersions with conformal Gauss
    map $V$.  In this case, $V$ is harmonic and $f,\hat{f}$
    are a dual pair of Willmore, thus $S$-Willmore
    \cite{Eji88}, surfaces.
  \item Where $\rank U=0$, $V$ is constant and there are
    infinitely many real, null line subbundles $f\leq V$
    defining conformal immersions with conformal Gauss map
    $V$.
  \end{compactenum}
\end{thm}
\begin{proof}
  First note that hypotheses \ref{item:1} and \ref{item:2}
  amount to the assertion that $U\leq V$ is isotropic so
  that $\rank U\leq 2$.

  We now consider each possibility for $\rank U$ in turn.

  First suppose that $\rank U=2$.  Then $U$ is maximal
  isotropic in $V$ and is $\cD^{0,1}$-stable by \cref{th:2}
  in view of \cref{th:6}.  By construction
  $\cN_{Z}V^{\perp}\sub U$ so that we may take $U=W$ in
  \cref{th:5} to learn that $V$ is the conformal Gauss map
  of $f=U\cap \overline{U}$ where the latter immerses.

  Now suppose that $\rank U=1$.  We claim that
  $U=\cN_ZV^{\perp}$: first this holds on a dense open set
  $\Omega$, (if $\cN_ZV^{\perp}$ vanishes on an open set, so
  does $\tau_V$) so that, by hypothesis \ref{item:3}, we
  have $U\cap \overline{U}=\set0$ on $\Omega$.  Since
  $\tau_V$ is real, we must have $\tau_V=0$ on $\Omega$ and
  hence everywhere so that the claim follows and $V$ is a
  harmonic map. It is now immediate that $U$ is
  $\cD^{0,1}$-stable.  By hypothesis \ref{item:3}, we have
  that $U\cap \overline{U}=\set0$ everywhere so that there
  are exactly two real, null line subbundles $f_1,f_2\leq V$
  orthogonal to $U\oplus \overline{U}$ and we set
  $W_i=f_i\oplus U$, $i=1,2$.  \cref{th:2}, applied to a
  section $w$ of $U$ assures us that each $W_i$ is
  $\cD^{0,1}$-stable so that \cref{th:5} gives that each
  $f_i$ is conformal where it immerses with conformal Gauss
  map $V$.  In this case, the $f_i$ are dual Willmore surfaces.

  Finally, if $\cN^{1,0}=0$ then $\cN$ vanishes also so that
  $V$ is $\d$-stable and so constant.  Thus $S^2:=\P(\cL\cap V)$
  is a conformal $2$-sphere and any conformal immersion
  $f:\Sigma\to S^2$ (in particular, any meromorphic function
  on $\Sigma$, off its branch locus) has $V$ as conformal
  Gauss map.
\end{proof}

\begin{rems}
\item[]
  \begin{compactenum}
  \item The caveat that $f$ immerse is not vacuous: one can
    readily construct $V$ satisfying the hypotheses of
    \cref{th:7} for which $f$ we find is constant.  Indeed,
    given constant $f\in\P(\cL)$, let $W\leq\ctriv$ be a
    non-constant rank $2$ isotropic subbundle containing $f$
    with $\overline{W}$ holomorphic with respect to the
    trivial holomorphic structure of $\ctriv$ and choose
    $V^{\perp}$ to be a complement to $W+\overline{W}$ in
    $f^{\perp}$.  Then it is not difficult to show that $W$
    is $\cD^{0,1}$-stable and $\cN^{1,0}W=\set0$.
  \item For strongly conformal $V$, equations \eqref{eq:7}
    are not independent.  Indeed, when $\rank\cN^{1,0}\restr{V^{\perp}}=2$,
    $\cN_ZV^{\perp}$ is maximal isotropic in $V$ so that
    \eqref{eq:8} forces
    $\tau_VV^{\perp}\leq\cN_{Z}V^{\perp}$.  Thus
    $\tau_VV^{\perp}$ is isotropic and \eqref{eq:9} holds.
    Again, when $\rank\cN^{1,0}\restr{V^{\perp}}=1$, it is
    easy to see that \eqref{eq:8} holds automatically.
  \end{compactenum}
\end{rems}

In the interesting case of harmonic $V$ (so that
$\tau_V=0$), matters simplify considerably.  Here, of
course, hypothesis \ref{item:2} of \cref{th:7} is vacuous.
Moreover, $\cN^{1,0}$ is a holomorphic $1$-form with respect
to the Koszul--Malgrange holomorphic structure of $\ctriv$
with $\bar{\partial}$-operator $\cD^{0,1}$.  It follows that
$\cN_{Z}\restr{V^{\perp}}$ has constant rank off a divisor
and, moreover, that there is a $\cD^{0,1}$-holomorphic
subbundle of $\ctriv$ that coincides with $\cN_ZV^{\perp}$
away from that divisor.  In this setting, we conclude with
Dorfmester--Wang:

\begin{cor}[c.f.\
  \citelist{\cite{DorWan13}*{Theorem~3.11}\cite{DorWan}*{Theorem~3.11}}]
  \label{th:8}
  Let $V\leq\triv$ be a strongly conformal harmonic bundle
  of $(3,1)$-planes.

  Let $U\leq V$ be the $\cD^{0,1}$-holomorphic, isotropic
  bundle that coincides with $\cN_ZV^{\perp}$ off a divisor. 
  \begin{compactenum}[(a)]
  \item if $\rank U=2$, there is a unique real, null line
    subbundle $f\leq V$ which, where it immerses, is a
    Willmore, non $S$-Willmore, surface with conformal Gauss
    map $V$.
  \item if $\rank U=1$ and $U\cap \overline{U}=\set0$, there
    are exactly two real, null line subbundles
    $f,\hat{f}\leq V$, which, where they immerse, are a dual
    pair of $S$-Willmore surfaces.
  \end{compactenum}
\end{cor}
\begin{rem}
  In the notation of Dorfmeister--Wang
  \cites{DorWan13,DorWan}, after a gauge transformation that
  renders $V,V^{\perp}$ constant, $\cN^{1,0}$ is represented
  by the matrix $B_1$.
\end{rem}





\end{document}